\documentclass[11pt]{amsart}
\pdfoutput=1
\usepackage{amsfonts,amsthm,amsmath}
\usepackage{natbib}
\usepackage{graphicx}
\usepackage{slashbox}

\newtheorem{theorem}{Theorem}[section]
\newtheorem{lemma}{Lemma}[section]
\newtheorem{prop}{Proposition}[section]

\newtheorem{remark}{Remark}

\newcommand{\vague}{\stackrel{v}{\longrightarrow}}

\newcommand{\prob}{\stackrel{P}{\longrightarrow}}
\newcommand{\eid}{\stackrel{d}{=}}
\newcommand{\one}{{\bf 1}}
\newcommand{\reals}{{\mathbb R}}
\newcommand{\bbr}{\reals}
\newcommand{\vep}{\varepsilon}

\newcommand{\bbrdcomp}{\overline{\bbr^d}}

\def\Var{{\rm Var}}

\numberwithin{equation}{section}

\begin{document}

\title[Truncated heavy tails]{Understanding heavy tails in a bounded
world or, is a truncated heavy tail heavy or not?}
\author[A. Chakrabarty]{Arijit Chakrabarty}
\address{School of Operations Research and Information Engineering,
Cornell University,
Ithaca, NY 14853, U.S.A.}
\email{ac427@cornell.edu}
\author[G. Samorodnitsky]{Gennady Samorodnitsky}
\address{School of Operations Research and Information Engineering,
Cornell University,
Ithaca, NY 14853, U.S.A.}
\email{gennady@orie.cornell.edu}

\begin{abstract}
We address the important question of the extent to which random
variables and vectors with truncated power tails retain the
characteristic features of random variables and vectors with power
tails. We define two truncation regimes,  soft truncation regime
and hard truncation regime, and show that, in the soft truncation
regime, truncated power tails behave, in important respects, as if
no truncation took place. On the other hand, in the hard truncation
regime much of ``heavy tailedness'' is lost. We show how to
estimate consistently the tail exponent when the tails are
truncated, and suggest statistical tests to decide on whether the
truncation is soft or hard. Finally, we apply our methods to two
recent data sets arising from computer networks.
\end{abstract}

\keywords{ heavy tails, truncation, regular variation, Central Limit
  theorem, Hill estimator, consistency\vspace{.5ex}}
\thanks{Research partly supported by the ARO grant W911NF-07-1-0078 at
Cornell University. Gennady Samorodnitsky's research  was also
partly supported by a Villum Kann Rasmussen Visiting Professor
Grant at the University of Copenhagen and by Otto Moensted foundation
grant at Danish Technological University.}

\maketitle

\section{Introduction} \label{sec:intro}

Probability laws with power tails are ubiquitous in applications.
A good fit between empirical distribution of various quantities
of interest and distributions with power tails has been reported in
such diverse areas as human travel
(\cite{brockmann:hufnagel:geisel:2006}), earthquake
analysis (\cite{corral:2006}), animal science
(\cite{bartumeus:daluz:vishwanathan:catalan:2005})
and even in language (\cite{serrano:flammini:menczer:2009}).
It is also true that in many situations there
is a ``physical'' limit that prevents a quantity of interest from
taking an arbitrarily large value. The File Allocation Table (FAT)
used on most computer systems allows the largest file size to be 4GB
(minus one byte) (\cite{microsoft:154997});
 the greatest loss an insurance company is exposed
to by a single covered event is limited by its reinsurance
contract (see e.g. \cite{mikosch:2009}). Even the number of the
atoms in the universe is widely considered to be finite. It is
common in practice to combine these two facts together and use a
model that features power tails only in a truncated form; such
models are often referred to as {\it truncated
  L\'evy flights}, see e.g. \cite{scholtz:contreras:1998},
\cite{maruyama:murakami:2003} or \cite{zaninetti:ferraro:2008}. At
the first glance this leads to a situation where the power tails,
in a sense, completely disappear. The truncation may change
dramatically the behavior of the cumulative sums of observations
and it always changes dramatically the behavior of the cumulative
maxima of the observations. Yet it is precisely such patterns of
behavior for which a model with power tails is chosen in the first
place. This leads one to ask the natural question: {\bf to what
extent, if any, do phenomena well described by models with
truncated power tails retain the characteristic features of power
tails}?

Answering this question is not straightforward. We start by
pointing out that the level of truncation is linked to the
amount of observations one has at hand. This can be thought of in
different ways. First of all, finiteness of the sample is sometimes
taken as the source of the truncation, see
e.g. \cite{burroughs:tebbens:2001} or
\cite{barthelemy:bertolotti:wiersma:2008}.
 Secondly, both the
physical nature of the truncation bound and the available data can be
linked to a technological level. This is particularly transparent when
one models a phenomenon related to computer or communications
systems; see e.g. \cite{jelenkovic:1999} or
\cite{gomez:selman:crato:kautz:2000}. We
describe this situation as a sequence of models, each one with
truncated power tails or, in other words, as a triangular array
system, which we now proceed to define formally.

Let $F$ be a probability law on $\bbr^d$, $d\geq 1$,
with the following property. There exists a sequence $(b_n)$ with
$b_n\uparrow\infty$ and a non-null Radon measure $\mu$ on
$\bbrdcomp\setminus\{0\}$ with $\mu\bigl\{ \bbrdcomp\setminus
\bbr^d\}=0$, such that
\begin{equation} \label{e:power.tails}
nF\bigl( b_n^{-1}\cdot\bigr)\vague \mu(\cdot)
\end{equation}
vaguely in $\bbrdcomp\setminus\{0\}$. Here $\bbrdcomp$ is the
compactification of $\bbr^d$ obtained by adding to the latter a
ball of infinite radius centered at the origin. The measure $\mu$
has necessarily a scaling property: there exists $\alpha>0$ such
that for any Borel set $B\in \bbr^d$ and $c>0$, $\mu(cB) =
c^{-\alpha}\mu(B)$. We say that the probability law $F$ has
regularly varying tails with the tail exponent $\alpha$ (see
\cite{resnick:1987},
\cite{hult:lindskog:mikosch:samorodnitsky:2005}), and  we view $F$
as the law with non-truncated power tails. When studying the
extent to which the central limit theorem behavior is affected by
truncation (which is the main point of interest to us in the
present paper) we will assume that $0<\alpha<2$. This restriction
on the tail exponent $\alpha$ is precisely the one that guarantees
that $F$ is in the domain of attraction of an $\alpha$-stable law;
see e.g. \cite{rvaceva:1962}. Such a restriction on the values of
the tail exponent will not be necessary in other parts of the
paper.

For $n=1,2,\ldots$ (regarded both as the number of observations in
the $n$th row of the triangular array and the number of the model)
let $M_n>0$ denote the truncation level. The $n$th row of the
triangular array will consist of observations $X_{nj},\,
j=1,\ldots, n$, which we view as generated according to the
following mechanism:
\begin{equation}\label{e:the.model}
X_{nj}:=H_j\one\bigl(\|H_j\|\le  M_n\bigr)
+\frac{H_j}{\|H_j\|}(M_n+R_j)\one\bigl(\|H_j\|>M_n\bigr)\,,
\end{equation}
$j=1,\ldots, n, \, n=1,2,\ldots$.
Here $H_1,H_2,\ldots$ are i.i.d. random vectors in $\bbr^d$ with the
common law $F$ that has regularly
varying tails with a tail exponent $\alpha\in (0,2)$,
and $R_1,R_2,\ldots$ are an independent of
$H_1,H_2,\ldots$ sequence of i.i.d. nonnegative random variables. For
each $n=1,2,\ldots$ we view the observations $X_{nj},\, j=1,\ldots, n$
as having power tails that are truncated at level $M_n$.

We need to comment, at this point, on the role of the random variables
$R_1,R_2,\ldots$. One should view them as possessing light tails, even
exponentially decaying tails. In many cases taking these random
variables to be equal to zero with probability 1 is appropriate; in
other applications exponentially fast tapering off of the tails beyond
the truncation point has
been observed (see e.g. \cite{hong:rhee:kim:lee:chong:2008}). The
reader will notice that the results of this paper hold whenever the
tails of the random variables
$R_1,R_2,\ldots$ are only light enough, not necessarily exponentially
light. We have chosen to formulate our results in this way in order
to increase their generality, even though we are thinking of their
role in the model \eqref{e:the.model} as representing the
exponentially fast decaying tails.

Our approach to addressing the question ``to what extent do
models with truncated power tails retain the characteristic
features of power tails?'' lies in studying the effect of the
rate of growth of the truncation level $M_n$ on the asymptotic
properties of the triangular array defined in
\eqref{e:the.model}. Specifically, we introduce the following
definition. We will say that the tails in the model
\eqref{e:the.model} are
\begin{equation} \label{e:regimes}
\begin{array}{ll}
\text{truncated softly} & \text{if} \ \lim_{n\to\infty}nP\bigl( \|
H_1\|>M_n\bigr) = 0\,,\\
\text{truncated hard} & \text{if} \ \lim_{n\to\infty}nP\bigl( \|
H_1\|>M_n\bigr) = \infty\,.
\end{array}
\end{equation}
Clearly, an intermediate regime exists as well. Understanding the
various aspects of an intermediately truncated model is an interesting
theoretical question, that we largely leave aside in this paper, in
order to keep its size manageable (see, however, Remark
\ref{rk:interm} of Section \ref{sec:CLT}). For practical purposes the
soft and hard truncation regimes provide the main dichotomy.

We will use fairly
classical techniques in Section \ref{sec:CLT} below to show that,
as far as the behavior of the partial sums of the truncated heavy
tailed model \eqref{e:the.model} is concerned, observations with
softly truncated tails behave like heavy tailed random variables,
while observations with hard truncated tails behave like light
tailed random variables. It is, however, clear that, in practice,
the truncation level $M_n$ is not observed. Therefore, we set
before ourselves two tasks in this paper. The first one, is to
estimate the tail exponent $\alpha$ based on a sample of
observations with truncated power tails without knowing the
truncation level or, even, if the truncation is soft or hard. We
show how this can be accomplished in Section \ref{sec:Hill}. The
second task is to find out whether the tails in the sample are
truncated softly or hard. In Section \ref{sec:regime.test}, where
we suggest statistical procedures for testing the hypothesis of
the soft (correspondingly, hard) truncation regime against the
appropriate alternative. In Section \ref{sec:data.analysis} we
apply the statistical techniques of Section \ref{sec:regime.test}
to two recent data sets related to TCP connections in a large
computer network. The goal for that section is only to illustrate the statistical tests discussed in Section \ref{sec:regime.test}.

We finish this section by pointing out that some of the issues
related to models with power tails that have been ``tampered with'', 
have been addressed in the literature. The goals, points of view or 
the ways in which the tails are modified were different from what is
done in the present work. The paper
\cite{asmussen:pihlsgard:2005} discusses an application of
distributions with truncated power tails in queuing, and addresses
the question whether light tailed approximations or heavy
approximations work better in this situation. On the other hand, a
maximum likelihood estimation procedure of the tail exponent
$\alpha$ in a parametric model of truncated power tails
(specifically, the truncated Pareto distribution) is given in
\cite{aban:meerschaert:panorska:2006}. Estimation of the
tail exponent in randomly censored power models (where the tails
are not so much truncated, as contaminated) is discussed in
\cite{beirlant:guillou:dierckx:fils-villetard:2007} and
\cite{einmahl:fils-villetard:guillou:2008}. Finally, the Hill
estimator for censored data is discussed in
\cite{beirlant:guillou:2001}. The techniques for dealing with censored
tails are consistent with the earlier work of
\cite{csorgo:horvath:mason:1986}.

\section{A Central Limit Theorem for random vectors with truncated
  power tails}
\label{sec:CLT}

Consider the triangular array defined in \eqref{e:the.model}, where,
as we recall, the random vectors $H_1,H_2,\ldots$ have a distribution
with regularly
varying tails with a tail exponent $\alpha\in (0,2)$. This means that
these random vectors (or their law $F$) are in the domain of
attraction of some $\alpha$-stable law $\rho$ on $\bbr^d$ (see
\cite{rvaceva:1962}). That is, the partial sums $ S_n^{(H)} =
H_1+\ldots + H_n$, $n=1,2,\ldots$, converge in law, after appropriate
centering and scaling, to $\rho$. Defining the sums of the {\it truncated}
observations,
$$
S_n:=\sum_{j=1}^nX_{nj}, \ n=1,2,\ldots\,,
$$
we would like to know whether $\bigl( S_n,\, n=1,2,\ldots\bigr)$ still
converge in law, after suitable centering and scaling, to $\rho$. If
the answer is no, then we would like to know what do these sums of
random vectors with truncated power tails converge to. These questions
can be handled by the classical probabilistic tools, and the answer
turns out to depend exclusively on the truncation regime as defined in
\eqref{e:regimes}.

\subsection{Soft truncation regime: truncated heavy tails are still heavy}

We start with the situation where the truncation level $M_n$ grows
sufficiently fast with the sample size, so that the truncated power
tails model \eqref{e:the.model} is in the soft truncation
regime. Theorem \ref{t1} below shows that, in this case, the partial
sums of the random vectors with truncated heavy tails converge, when
properly centered and scaled, to the same $\alpha$-stable limit as
without truncation.

Let $(c_n)$ and $(b_n)$ denote, respectively, some centering and
scaling sequences for the non-truncated random vectors $(H_j)$, that is,
\begin{equation} \label{e:no.truncation}
b_n^{-1}S_n^{(H)} -c_n =
b_n^{-1}\sum_{j=1}^nH_j-c_n\Longrightarrow\rho
\end{equation}
as $n\to\infty$.

\begin{theorem}\label{t1} In the soft truncation regime we have
\begin{equation}\label{e:soft.same}
b_n^{-1}S_n-c_n\Longrightarrow\rho\,.
\end{equation}
\end{theorem}
\begin{proof} By \eqref{e:no.truncation} it is enough to show that
$$
b_n^{-1}\left\|S_n-\sum_{j=1}^nH_j\right\|\stackrel
p\longrightarrow0\,.
$$
However, for any $\vep>0$,
$$
P\left( b_n^{-1}\left\|S_n-\sum_{j=1}^nH_j\right\|>\vep\right)
\leq P\Bigl( \|H_j\|>M_n \ \ \text{for some $j=1,\ldots, n$}\Bigr)
$$
$$
\leq nP\bigl( \|H_1\|>M_n\bigr)\to 0\,,
$$
and the claim follows.
\end{proof}

\subsection{Hard Truncation regime: truncated heavy tails are no
  longer heavy}

Now we consider the situation where the truncation level $M_n$
grows relatively slowly with the sample size, and that the
truncated power tails model \eqref{e:the.model} is in the hard
truncation regime. As we will see, in this case the partial sums
of the random vectors with truncated heavy tails are no longer
asymptotically $\alpha$-stable but, rather, converge in law, after
suitable centering and scaling, to a Gaussian limit. Therefore, at
least from the point of view of the behavior of partial sums, a
model with power tails that have been truncated hard does not
behave anymore as a heavy tailed model.

We start with some preliminaries. Recall that, since the limiting
law $\rho$ in \eqref{e:no.truncation} is $\alpha$-stable, the
L\'evy-Khinchine formula for its characteristic function has the form
\begin{equation} \label{e:levy-khnichine}
\hat\rho(\theta)=\exp\biggl[i\langle
  \theta,\gamma\rangle
\end{equation}
$$
+\int_{S}\left( \int_0^\infty
\left\{e^{ix\langle
      \theta,s\rangle}-1-ix\langle \theta,s\rangle \one(x\le
    1)\right\}x^{-(1+\alpha)}\, dx\right)\Gamma(ds)\biggr]
$$
for $\theta\in\bbr^d$,
where $\gamma\in\bbr^d$, and $\Gamma$ is a finite measure on
the unit sphere in $\bbr^d$,
$S:=\{x\in {\mathbb R}^d:\|x\|=1\}$, see Theorem 6.15 in
\cite{araujo:gine:1980}. The measure $\Gamma$
is often referred to as spectral measure of the law $\rho$; see
Theorem 2.3.1 in \cite{samorodnitsky:taqqu:1994}.

\begin{theorem}\label{t2} Assume that $ER_1^2<\infty$, and let
$$
B_n:=\bigl[ nM_n^2P(\|H_1\|>M_n)\bigr]^{1/2},\, n=1,2,\ldots\,.
$$
Then in the hard truncation regime we have
\begin{equation}\label{e:hard.gauss}
B_n^{-1}(S_n-ES_n)\Longrightarrow\eta\,,
\end{equation}
where $\eta$ is a centered Gaussian law on ${\mathbb R}^d$ whose
covariance matrix has the entries
\begin{equation}\label{clt.e3}
\frac{2}{2-\alpha}\int_S s_is_j\,\tilde\Gamma(ds),\, i,j =1,\ldots, d\,,
\end{equation}
where
$\tilde\Gamma(\cdot):=\Gamma(\cdot)/\Gamma(S)$ is the normalized
spectral measure of $\rho$.
\end{theorem}

We start with a lemma.

\begin{lemma}\label{clt.l5} For every
  continuous function $f:S\longrightarrow\mathbb R$,
$$
\lim_{n\rightarrow\infty}nB_n^{-2}\int_S\int_0^{M_n}f(s)r^2P\left(\|H_1\|\in
  dr,\frac {H_1}{\|H_1\|}\in
  ds\right)=\frac\alpha{2-\alpha}\int_Sf(s)\, \tilde\Gamma(ds).
$$
\end{lemma}

\begin{proof} Assumption \eqref{e:no.truncation} means that
\begin{equation} \label{e:weak.conv}
\frac{P\left( \| H_1\|>r,\, \frac{H_1}{\| H_1\|}\in\cdot\right)}{P\bigl(
 \| H_1\|>r\bigr)}\Longrightarrow   \tilde\Gamma(\cdot)
\end{equation}
weakly on $S$; see e.g.  Corollary 6.20 (b) of
\cite{araujo:gine:1980}. Therefore,
$$
\int_S\int_0^{M_n}f(s)r^2P\left(\|H_1\|\in   dr,\frac {H_1}{\|H_1\|}\in
  ds\right)
$$
$$
= \int_0^{M_n} 2y\left( \int_S f(s)P\left(\|H_1\|>y,\,\frac {H_1}{\|H_1\|}\in
  ds\right)\right)dy
$$
$$
- M_n^2 \int_S f(s)P\left(\|H_1\|>M_n,\,\frac {H_1}{\|H_1\|}\in
  ds\right)
$$
$$
\sim \int_Sf(s)\, \tilde\Gamma(ds)\left[ \int_0^{M_n} 2yP\bigl(
  \|H_1\|>y\bigr)\, dy - M_n^2P\bigl( \|H_1\|>M_n\bigr)\right]
$$
$$
\sim \int_Sf(s)\, \tilde\Gamma(ds) \left( \frac{2}{2-\alpha}-1\right)
 M_n^2P\bigl( \|H_1\|>M_n\bigr) = n^{-1}B_n^2  \int_Sf(s)\,
 \tilde\Gamma(ds)
$$
as $n\to\infty$, where the second asymptotic equivalence follows from
the Karamata theorem (see e.g. \cite{resnick:1987}).
\end{proof}

\begin{proof}[Proof of Theorem \ref{t2}] By the Cram\'er-Wold device it
suffices to show that for every $\theta$ in ${\mathbb R}^d$,
$$
B_n^{-1}\bigl(\langle \theta,S_n\rangle-E\langle
\theta,S_n\rangle\bigr)\Rightarrow
N\left(0,\frac2{2-\alpha}\int_S\langle
  \theta,s\rangle^2\tilde\Gamma(ds)\right)\,.
$$
To this end we will use the Central Limit Theorem for triangular arrays
under the Lindeberg condition; see e.g. Theorem 2.4, page 345 in
\cite{gut:2005}. We need to prove that
\begin{equation} \label{e:var}
\lim_{n\to\infty}\frac{n}{B_n^2}\Var\bigl(\langle
\theta,X_{n1}\rangle\bigr)  = \frac2{2-\alpha}\int_S\langle
  \theta,s\rangle^2\tilde\Gamma(ds)
\end{equation}
and that for every $\vep>0$,
\begin{equation} \label{e:lindeberg}
 \frac{n}{B_n^2} E\biggl( \bigl| \langle
\theta,X_{n1}\rangle -
E\bigl(\langle\theta,X_{n1}\rangle\bigr)\bigr|^2
\one\Bigl( \bigl| \langle\theta,X_{n1}\rangle -
E\bigl(\langle\theta,X_{n1}\rangle\bigr)\bigr|>\vep B_n\Bigr)\biggr) \to
0
\end{equation}
as $n\to\infty$. In order to prove \eqref{e:var}, we will show that
\begin{equation} \label{e:2ndmom}
\lim_{n\to\infty}\frac{n}{B_n^2}E\bigl((\langle
\theta,X_{n1}\rangle)^2\bigr)  = \frac2{2-\alpha}\int_S\bigl(\langle
  \theta,s\rangle\bigr)^2\tilde\Gamma(ds)
\end{equation}
while
\begin{equation} \label{e:1stmom}
\lim_{n\to\infty}\frac{n^{1/2}}{B_n}\bigl|E\bigl(\langle
\theta,X_{n1}\rangle\bigr)\bigr|  = 0\,.
\end{equation}
The former claim follows easily from Lemma \ref{clt.l5} and the weak
convergence \eqref{e:weak.conv} by writing
$$
E\bigl((\langle\theta,X_{n1}\rangle)^2\bigr)  =
E\Bigl((\langle\theta,H_{1}\rangle)^2\one\bigl(\|H_1\|\le
M_n\bigr)\Bigr)
$$
$$
+ E\left(
  \frac{\bigl(\langle\theta,H_{1}\rangle\bigr)^2}{\|H_1\|^2}
  (M_n+R_1)^2 \one\bigl(\|H_1\|> M_n\bigr)\right)
$$
$$
\sim n^{-1}B_n^{2} \frac\alpha{2-\alpha}\int_S\bigl(\langle
  \theta,s\rangle\bigr)^2\, \tilde\Gamma(ds)
+ \bigl( 1+o(1)\bigr) M_n^2  E\left(
  \frac{\bigl(\langle\theta,H_{1}\rangle\bigr)^2}{\|H_1\|^2}
\one\bigl(\|H_1\|> M_n\bigr)\right)
$$
$$
\sim n^{-1}B_n^{2} \frac\alpha{2-\alpha}\int_S\bigl(\langle
  \theta,s\rangle\bigr)^2\, \tilde\Gamma(ds)
+ M_n^2P\bigl(\|H_1\|> M_n\bigr) \int_S\bigl(\langle
  \theta,s\rangle\bigr)^2\, \tilde\Gamma(ds)
$$
$$
= n^{-1}B_n^{2} \left( \frac\alpha{2-\alpha}+1\right)
\int_S\bigl(\langle
  \theta,s\rangle\bigr)^2\, \tilde\Gamma(ds)
= n^{-1}B_n^{\alpha}\frac\alpha{2-\alpha} \int_S\bigl(\langle
  \theta,s\rangle\bigr)^2\, \tilde\Gamma(ds)\,.
$$
For \eqref{e:1stmom} we write
$$
\bigl|E\bigl(\langle \theta,X_{n1}\rangle\bigr)\bigr|
\leq \|\theta\| \Bigl[ E\bigl( \| H_1\|\one( \| H_1\|\leq M_n)\bigr)
+ M_nP\bigl( \| H_1\|>M_n\bigr)\Bigr]\,.
$$
Since
$$
M_nP\bigl( \| H_1\|>M_n\bigr) \ll M_n\bigl( P( \|
H_1\|>M_n)\bigr)^{1/2}
= n^{-1/2}B_n\,,
$$
the claim \eqref{e:1stmom} will follow once we check that
\begin{equation}\label{t2.e10}
\lim_{n\rightarrow\infty}n^{1/2}B_n^{-1}E\left[\|H_1\|\one(\|H_1\|\le
    M_n)\right]=0 \,.
\end{equation}
We give separate arguments for the cases $\alpha\le 1$ and $\alpha>1$.\\
{\bf Case 1 ($\alpha\le1$):} Letting $C$ be a positive
constant whose value may change from line to line, by the Karamata
theorem,
\begin{eqnarray*}
E\left[\|H_1\|\one(\|H_1\|\le M_n)\right]
&\le&\bigl( E\left[\|H_1\|^{3/2}\one(\|H_1\|\le M_n)\right]\bigr)^{2/3}\\
&\sim&CM_n\bigl( P(\|H_1\|>M_n)\bigr)^{2/3}\\
&=&Cn^{-1/2}B_n\bigl( P(\|H_1\|>M_n)\bigr)^{1/6}\\
&\ll&n^{-1/2}B_n\,.
\end{eqnarray*}

{\bf Case 2 ($1<\alpha<2$):} Here (\ref{t2.e10}) follows trivially
from the fact that $E\left[\|H_1\|\one(\|H_1\|\le
    M_n)\right]$ has a finite limit, while
$B_n\gg n^{1/2}$ as $\alpha<2$.

We have now proved \eqref{e:var}. By \eqref{e:1stmom},
the remaining condition \eqref{e:lindeberg} will follow once we check
that for every $\vep>0$,
$$
\frac{n}{B_n^2}E\Bigl( \bigl| \langle
\theta,X_{n1}\rangle \bigr|^2 \one\bigl( \bigl|
\langle\theta,X_{n1}\rangle\bigr| >\vep B_n\bigr)\Bigr)\to 0\,.
$$
This is, however, an immediate consequence of the fact that the hard
truncation implies that $B_n\gg M_n$ as $n\to\infty$.
\end{proof}

\begin{remark} \label{rk:interm}
{\rm We briefly address the behavior of the partial sums of the
random vectors with truncated heavy tails in the intermediate
regime
\begin{equation} \label{e:interm}
\lim_{n\rightarrow\infty}nP(\|H\|>M_n)=\delta\in(0,\infty)\,.
\end{equation}
It turns out that, in this case, one can use the same centering and
scaling sequences $\{c_n\}$ and $\{b_n\}$ as in the non-truncated case
\eqref{e:no.truncation} (or in the soft truncation regime
\eqref{e:soft.same}), but the limit will be different. In fact,
\begin{equation} \label{e:interm.limit}
b_n^{-1}S_n-c_n\Longrightarrow \rho_\delta\,,
\end{equation}
where $\rho_\delta$ is an infinitely divisible law on $\bbr^d$, which
is obtained by a certain truncation of the jumps of the
$\alpha$-stable law $\rho$ in \eqref{e:levy-khnichine}. Specifically,
\begin{equation} \label{e:rho.delta}
\hat\rho_\delta(\theta)=\exp\biggl[i\langle
\theta,\gamma_\delta\rangle
\end{equation}
$$
+\int_{S}\biggl(
\int_0^{\delta^{-1/\alpha}(\alpha^{-1}\Gamma(S))^{1/\alpha}}
\left\{e^{ix\langle
      \theta,s\rangle}-1-ix\langle \theta,s\rangle \one(x\le
    1)\right\}x^{-(1+\alpha)}\, dx
$$
$$
+ \delta\Gamma(S)^{-1} \left\{ e^{i
\delta^{-1/\alpha}(\alpha^{-1}\Gamma(S))^{1/\alpha}
\langle \theta,s\rangle}-1\right\}
\biggr)\Gamma(ds)\biggr]
$$
for $\theta\in\bbr^d$, where
$$
\gamma_\delta = \gamma -
\int_{\delta^{-1/\alpha}(\alpha^{-1}\Gamma(S))^{1/\alpha}}^\infty
x^{-\alpha} \one(x\le 1)\, dx \int_S s\, \Gamma(ds)\,.
$$
We sketch the argument. Write
\begin{equation}\label{e:interm.1}
b_n^{-1}S_n-c_n =
\left(b_n^{-1}\sum_{j=1}^nH_j\one\bigl(\|H_j\|\le M_n\bigr)
-c_n\right)
\end{equation}
$$
+  b_n^{-1}M_n \sum_{j=1}^n
\frac{H_j}{\|H_j\|}\one\bigl(\|H_j\|>M_n\bigr)
+ b_n^{-1} \sum_{j=1}^n
\frac{H_j}{\|H_j\|}R_j\one\bigl(\|H_j\|>M_n\bigr)\,.
$$
It is easy to check that the last term in the right hand side of
\eqref{e:interm.1} converges to zero in probability. Since
\eqref{e:interm} implies that
\begin{equation}\label{e:ratio}
\frac{M_n}{b_n} \to
\delta^{-1/\alpha}(\alpha^{-1}\Gamma(S))^{1/\alpha}
\end{equation}
as $n\to\infty$,  Theorem 5.9, p. 129, of \cite{araujo:gine:1980}
implies that the first term in the right hand side of
\eqref{e:interm.1} has a weak limit whose characteristic function is
given by
$$
\exp\biggl[i\langle
\theta,\gamma_\delta\rangle
$$
$$
+
\int_{S}\biggl(
\int_0^{\delta^{-1/\alpha}(\alpha^{-1}\Gamma(S))^{1/\alpha}}
\left\{e^{ix\langle
      \theta,s\rangle}-1-ix\langle \theta,s\rangle \one(x\le
    1)\right\}x^{-(1+\alpha)}\, dx \biggr)\Gamma(ds)\biggr]\,.
$$
Finally, it follows from \eqref{e:ratio} that the second term in
the right hand side of \eqref{e:interm.1} is asymptotically
equivalent to
$$
\delta^{-1/\alpha}(\alpha^{-1}\Gamma(S))^{1/\alpha}
\sum_{j=1}^n
\frac{H_j}{\|H_j\|}\one\bigl(\|H_j\|>M_n\bigr)\,,
$$
and, by \eqref{e:weak.conv} and \eqref{e:interm},
 the sum above converges weakly to the law of the Poisson sum
$\sum_{j=1}^N Y_j$, where $Y_1,Y_2,\dots$ are i.i.d. $S$-valued random
variables with the common law $\tilde\Gamma$, and $N$ is an independent
of them Poisson random variable with mean $\delta$. Since the weak
limits of the first and the second terms in the right hand side of
\eqref{e:interm.1} are easily seen to be independent, this shows
\eqref{e:interm.limit}.
}
\end{remark}

\section{Hill estimator for random variables with truncated power tails} \label{sec:Hill}

Estimating the tail exponent $\alpha$ is one of the main
statistical issues one faces when working with data for which a
model with power tails is contemplated. This is a difficult
statistical problem because one attempts to estimate a parameter
governing the tail behavior in an otherwise nonparametric model.
By necessity, any estimator one uses has to be based on a
vanishing fraction of the available data. The situation is even
trickier when one tries to estimate the tail exponent in a sample
of observations with truncated power tails. This is the task we
address in this section.

The formal setup in this section is as follows. We are given a
sample $X_1,\ldots, X_n$ of {\bf one-dimensional nonnegative
observations} from the model with truncated power tails, i.e.
\eqref{e:the.model}. We emphasize a slight change in notation from
\eqref{e:the.model}: whereas the latter used the notation
$X_{n1},\ldots, X_{nn}$ to emphasize the triangular array nature
of the model, in a statistical procedure, when a single sample
({\it i.e.}, a particular row of the triangular array) is given,
the notation $X_1,\ldots, X_n$ is more natural. The discussion in
Section \ref{sec:CLT} makes it intuitive that estimating the tail
exponent $\alpha$ should be easier if the tails are truncated
softly, than in the case when the tails are truncated hard. This
is, indeed, the case. However, in this section we are interested
in finding a procedure that permits consistent estimation of the
tail exponent $\alpha$ regardless of the truncation regime; this
is especially important because the truncation regime is never
known (see, however, Section \ref{sec:regime.test} below).
Furthermore, in this section we do not restrict the values of the
tails exponent to the interval $(0,2)$. That is, $\alpha$ can take
any positive value.

A number of estimators of the tail exponent of distributions with
nontruncated power tails
have been suggested; a thorough discussion can
be found in Chapter 4 of \cite{dehaan:ferreira:2006}. One of the best
known and widely used estimators is the {\it Hill estimator}
introduced by \cite{hill:1975}. Given a sample $X_1,\ldots, X_n$. the
Hill statistic is defined by
\begin{equation} \label{e;hill}
h_{n,k} = \frac1{k}\sum_{i=1}^{k}\log\frac{X_{(i)}}{X_{(k)}}\,,
\end{equation}
where $X_{(1)}\ge X_{(2)}\ge\ldots\ge X_{(n)}$ are the order
statistics from the sample $X_{1},\ldots,X_{n}$, and $k=1,\ldots, n$
is a user-determined parameter, the number of the upper order
statistics to use in the estimator. The consistency result for the
Hill estimator says that, if $X_1,\ldots, X_n$ are i.i.d. with
regularly varying right tail with exponent $\alpha>0$, and
$k=k_n\to\infty$, $k_n/n\to 0$ as $n\to\infty$, then $h_{n,k_n} \to
1/\alpha$ in probability as $n\to\infty$; see e.g. Theorem 3.2.2 in
\cite{dehaan:ferreira:2006}.

In spite of the simplicity of the  statement of the consistency of
the Hill estimator, selecting the  number $k$ of the upper order
statistics for a given sample with nontruncated power tails
remains a daunting problem; see e.g. pp. 192-193 in
\cite{embrechts:kluppelberg:mikosch:1997}. In the main result of
this section, Theorem \ref{hill.t1} below, we will see that one
has to be particularly careful when using the Hill estimator on a
sample with truncated power tails. Nonetheless, a consistent
estimator can still be obtained.

Notice that the next theorem does not impose any conditions on the
random variables $R_1,R_2,\ldots$ in the model
\eqref{e:the.model}.

\begin{theorem}\label{hill.t1}
Suppose that the number $k_n$ of the  upper order statistics satisfies
\begin{equation} \label{e:hill.range}
nP(H_1>M_n)+1\ll k_n\ll n\,.
\end{equation}
Then $h_{n,k_n} \to 1/\alpha$ in probability as $n\to\infty$.
\end{theorem}
Note that Theorem \ref{hill.t1} says that, in the soft truncation
regime, the Hill estimator is consistent under the same assumption,
$k_n/n\to 0$, as in the nontruncated case.
\begin{proof}
For simplicity, we write $k$ instead of $k_n$.
An inspection of the proof of consistency of the Hill estimator in the
nontruncated case in e.g. \cite{resnick:2007} shows that the result
will follow once we check that, under the conditions of the theorem,
\begin{equation}\label{hill.l1.e1}
\frac nkP\left[\frac{X_{n1}}{b(n/k)}\in\cdot\right]\stackrel
 v\longrightarrow\mu(\cdot)
\end{equation}
vaguely in $(0,\infty]$, where $\mu$ is a measure on
$(0,\infty]$ defined by
$$
\mu((x,\infty])=x^{-\alpha}\mbox{ for all }x>0\,,
$$
and
$$
b(u) = \inf\bigl\{ x>0:\, P(H_1>x)\leq u^{-1}\bigr\},\,
u>0\,.
$$
Note that $(b(n))$ is no longer necessarily a sequence satisfying
\eqref{e:no.truncation}. In fact, we will use this notation several
times in the sequel to denote other quantile-type functions associated
with the random variable $H_1$.

By the hypothesis,
$$
\lim_{n\rightarrow\infty}\frac nkP(H_1>M_n)=0
$$
and, hence, $b(n/k)\ll M_n$ as $n\to\infty$. Therefore, for any $x>0$,
for $n$ large enough,
\begin{eqnarray*}
P\left[\frac{X_{n1}}{b(n/k)}>x\right]&=&P\bigl( H_1>xb(n/k)\bigr)\\
&\sim&\frac knx^{-\alpha}
\end{eqnarray*}
where the last line follows from the hypothesis $k\ll n$ and regular
variation of the tail of $H_1$.  This shows (\ref{hill.l1.e1}).
\end{proof}

Since the truncation level $M_n$ is not known, it is desirable to have
a sample-based way of deciding on the number of upper order statistics
to use in the Hill estimator. A natural (in view of the condition
\eqref{e:hill.range}) choice is to use {\it a random number} of upper
order statistics given by
\begin{equation} \label{e:k.random}
{\hat k}_n = \left[n\left( \frac1n \sum_{j=1}^n \one\bigl(
  X_j>\gamma\max_{i=1,\ldots, n} X_i\bigr)\right)^\beta\right]\,,
\end{equation}
where $\gamma$ and $\beta$ are user-specified parameters taking values
in $(0,1)$ and $[\cdot]$ denotes the integer part. The  next two
theorems show that this choice 
of the number of upper order statistics leads to a consistent
estimator of the reciprocal of the tail exponent. The parameters
$\beta$ and $\gamma$ can and should be chosen by the user, and the
consistency of the estimator depends neither on their choice nor on,
for instance, second order regular variation of $H_1$. In practice,
one should, probably, use some version of the so-called Hill plot (see
e.g. \cite{embrechts:kluppelberg:mikosch:1997}), by
calculating the value of the estimator for a range of $\beta$ and
$\gamma$. We leave this issue for future work. 

We start with proving consistency of the Hill estimator with the
random choice of the number of upper order statistics in the
hard truncation regime. Here the consistency requires that the tail
of the random variable $R_1$ is sufficiently light, relatively to the
truncation level. 

\begin{theorem}\label{hill.t5}
In the hard truncation regime assume, additionally, that 
\begin{equation}\label{assume1}
P(H_1>M_n)P(R_1>\epsilon M_n)=o(1/n)\ \mbox{ for all }\ \epsilon>0.
\end{equation}
If $\hat k_n$ is chosen as in \eqref{e:k.random}, then $h_{n,\hat k_n}
\to 1/\alpha$ in probability as $n\to\infty$. 
\end{theorem}

\begin{proof}%[Proof of Theorem \ref{hill.t5}]
We start by showing  that, as $n\longrightarrow\infty$,
\begin{equation}\label{hill.t5.eq3}
\frac{\hat k_n}{k_n}\prob1\,,
\end{equation}
where
$$
k_n:=\left[nP(H_1>\gamma M_n)^\beta\right]\,.
$$
Clearly, all that needs to be shown is
\begin{equation}\label{l1.2}
\frac{\sum_{j=1}^n\one(X_{j}>\gamma\hat M_n)}{nP(H_1>\gamma
  M_n)}\stackrel P\longrightarrow1, 
\end{equation}
where $\hat M_n:=\max_{1\le i\le n}X_{i}$. By (\ref{assume1}), 
\begin{equation}\label{l1.3}
\frac{\hat M_n}{M_n}\stackrel P\longrightarrow1\,.
\end{equation}
Further, because of the hard truncation, 
\begin{equation}\label{l1.1}
\frac{\sum_{j=1}^n\one\left(j:X_{j}>\theta_1M_n\right)}{nP(H>\theta_2M_n)}
\stackrel
P\longrightarrow\left(\frac{\theta_1}{\theta_2}\right)^{-\alpha}  
\end{equation}
for any $0<\theta_1<1$ and $\theta_2>0$. 

Fix $0<\epsilon<1$. Choose first $0<\eta<1$ be such that
$(1-\eta)^{-\alpha}<1+\epsilon$. Writing 
\begin{eqnarray*}
&&P\left[\frac{\sum_{j=1}^n\one\left(X_{j}>\gamma \hat
      M_n\right)}{nP(H_1>\gamma  M_n)}>1+\epsilon\right]\\ 
%&\le&P\left[\frac{\sum_{j=1}^n\one\left(X_{j}>\gamma \hat M_n\right)}{\sum_{j=1}^n\one\left(X_{j}>\gamma (1-\eta) M_n\right)}>1\right]\\
%&&+P\left[\frac{\sum_{j=1}^n\one\left(X_{j}>\gamma (1-\eta) M_n\right)}{nP(H>\gamma  M_n)}>1+\epsilon\right]\\
&\le&P[(1-\eta)M_n>\hat
    M_n]+P\left[\frac{\sum_{j=1}^n\one\left(X_{j}>\gamma (1-\eta)
      M_n\right)}{nP(H_1>\gamma  M_n)}>1+\epsilon\right]
\end{eqnarray*}
and using (\ref{l1.3}) and (\ref{l1.1}), we see that 
%$$\frac{\sum_{j=1}^n\one\left(X_{j}>\gamma (1-\eta) M_n\right)}{nP(H>\gamma  M_n)}\stackrel P\longrightarrow(1-\eta)^{-\alpha}<1+\epsilon$$
%and hence
$$
P\left[\frac{\sum_{j=1}^n\one\left(X_{j}>\gamma \hat
    M_n\right)}{nP(H_1>\gamma
    M_n)}>1+\epsilon\right]\longrightarrow0\,.
$$ 
Next we choose $\eta>0$ such that $\gamma(1+\eta)<1$ and
$(1+\eta)^{-\alpha}>1-\epsilon$. Writing 
\begin{eqnarray*}
&&P\left[\frac{\sum_{j=1}^n\one\left(X_{j}>\gamma \hat
      M_n\right)}{nP(H_1>\gamma  M_n)}<1-\epsilon\right]\\ 
%&\le&P\left[\frac{\sum_{j=1}^n\one\left(X_{j}>\gamma \hat M_n\right)}{\sum_{j=1}^n\one\left(X_{j}>\gamma (1+\eta) M_n\right)}<1\right]\\
%&&+P\left[\frac{\sum_{j=1}^n\one\left(X_{j}>\gamma (1+\eta) M_n\right)}{nP(H>\gamma  M_n)}<1-\epsilon\right]\\
&\le&P[(1+\eta)M_n<\hat
    M_n]+P\left[\frac{\sum_{j=1}^n\one\left(X_{j}>\gamma (1+\eta)
      M_n\right)}{nP(H_1>\gamma  M_n)}<1-\epsilon\right]\,, 
\end{eqnarray*}
and appealing, once again, to (\ref{l1.3}) and (\ref{l1.1}), we obtain 
\begin{eqnarray*}
P\left[\frac{\sum_{j=1}^n\one\left(X_{j}>\gamma \hat
    M_n\right)}{nP(H_1>\gamma  M_n)}<1-\epsilon\right]\longrightarrow0\,, 
\end{eqnarray*}
which establishes (\ref{l1.2}), and, hence, also \eqref{hill.t5.eq3}.

In view of  \eqref{hill.t5.eq3}, it suffices to show that
$$
\frac1{k_n}\sum_{i=1}^{\hat k_n}\log\frac{X_{(i)}}{X_{(\hat
    k_n)}}\stackrel P\longrightarrow\frac1\alpha\,.
$$
Fix $\epsilon>0$. We choose $0<\eta<1/2$ so that
$\alpha^{-1}\log\frac{1+\eta}{1-\eta}<\frac\epsilon3$, and write 
\begin{eqnarray*}
&&P\left[\left|\frac1{k_n}\sum_{i=1}^{\hat
      k_n}\log\frac{X_{(i)}}{X_{(\hat
        k_n)}}-\frac1\alpha\right|>\epsilon\right]\\ 
&\le&P\left[\left|\frac1{k_n}\sum_{i=1}^{k_n}\log\frac{X_{(i)}}{X_{(k_n)}}
-\frac1\alpha\right|>\frac\epsilon3\right] 
+P\left[-\log\frac{X_{([k_n(1+\eta)])}}{X_{([k_n(1-\eta)])}}>\frac\epsilon3\right]\\ 
&&+P\left[\left|\frac{\hat k_n}{k_n}-1\right|\ge\eta\right]\,.
\end{eqnarray*}
It follows from Theorem \ref{hill.t1} that
$$
P\left[\left|\frac1{k_n}\sum_{i=1}^{k_n}\log\frac{X_{(i)}}{X_{(k_n+1)}}
-\frac1\alpha\right|>\frac\epsilon3\right]\longrightarrow0\,.
$$ 
Since $k_n\gg nP(H_1>M_n)$, we see that 
$$
\frac{X_{([k_n(1+\eta)])}}{X_{([k_n(1-\eta)])}}\stackrel
P\longrightarrow\left(\frac{1+\eta}{1-\eta}\right)^{-1/\alpha}\,.
$$ 
Therefore, by the choice of $\eta$, 
$$
P\left[-\log\frac{X_{([k_n(1+\eta)])}}{X_{([k_n(1-\eta)])}}>\frac\epsilon3\right]
\longrightarrow0\,.
$$ 
In conjunction with \eqref{hill.t5.eq3}, this completes the argument. 
\end{proof}

Next, we show consistency of the Hill estimator using the random
number \eqref{e:k.random} of upper order statistics in the soft
truncation regime. Note that no 
assumption on the tail of $R_1$ is necessary in this case.

\begin{theorem}\label{hill.t6} In the soft truncation regime, if $\hat
  k_n$ is chosen as in \eqref{e:k.random}, then $h_{n,\hat k_n} 
\to 1/\alpha$ in probability as $n\to\infty$. 
\end{theorem}

\begin{proof} Let $\bigl(\tilde h(k,n)\bigr)$ denote the Hill sequence
based on the random variables   $H_1,\ldots,H_n$, i.e. 
$$
\tilde h(k,n):=\frac1k\sum_{i=1}^k\log\frac{H_{(i)}}{H_{(k)}}\,,
$$
where $1\le k\le n$ and $H_{(1)}\ge\ldots\ge H_{(n)}$ are the order
statistics of $H_1,\ldots,H_n$. Let $\beta$ be as in
\eqref{e:k.random}. It is well known that the random step function 
$$
\left(\tilde h([n^{1-\beta}t],n):t\ge1\right)
$$ 
converges in probability in $D[1,\infty)$, equipped with the topology
  of uniform   convergence on bounded intervals, to the constant 
deterministic function $c(t) = 1/\alpha$ for $t\geq 1$; see
\cite{resnick:2007}, page 89. Since 
$$
P\left(\tilde h([n^{1-\beta}t],n)\neq h([n^{1-\beta}t],n)\mbox{ for
  some }t\ge1\right)\le nP(H>M_n)\longrightarrow0\,, 
$$
it follows that 
\begin{equation}\label{hill.t6.eq1}
\left(h([n^{1-\beta}t],n):t\ge1\right)\prob \bigl( c(t): t\geq 1\bigr) 
\end{equation}
in $D[1,\infty)$ as well. 
By Proposition 3.21 (page 154) in \cite{resnick:1987} and soft
truncation, we known that 
$$
N_n:=\sum_{j=1}^n \one\bigl(
  X_j>\gamma\max_{i=1,\ldots, n} X_i\bigr)\Longrightarrow
  N:=\sum_{j=1}^\infty\one\left(\Gamma_j^{-1/\alpha}>\gamma
\Gamma_1^{-1/\alpha}\right)\,, 
$$
where $\left(\Gamma_j:j\ge1\right)$ is the sequence of the arrival
times of the unit rate Poisson process on $(0,\infty)$. Using 
\eqref{hill.t6.eq1} and Theorem 4.4 (page 27) in
\cite{billingsley:1968} we conclude that 
$$
\left(h([n^{1-\beta}t],n),\, N_n^\beta\right)\Longrightarrow
\left(\frac1\alpha,\, N^\beta\right) 
$$
in $D[1,\infty)\times{\mathbb N}^\beta$, where
$
{\mathbb N}^\beta:=\{1,2^\beta,3^\beta,\ldots\}
$.
Since the evaluation map from $D[1,\infty)\times{\mathbb N}^\beta$ to
  $\bbr$ defined by $(x,a)\mapsto x(a)$ 
is continuous, an appeal to the continuous mapping theorem finishes 
the proof. 
\end{proof}

\section{Testing for soft and hard truncation} \label{sec:regime.test}

The first two sections of this paper provide, among other things,
evidence that, in certain important respects, random variables with
truncated heavy tails retain ``most of the tail heaviness'' if the
truncation is soft, but loose ``much of the tail heaviness'' if the
truncation is hard. Since the truncation level is not observed, how
does one decide if the tails of observed data have been truncated softly
or hard? In this section we construct statistical tests for testing
each of the two hypothesis against the corresponding alternative. 
As in Section \ref{sec:Hill} we restrict ourselves to the case of
one-dimensional observations and the tail
exponent $\alpha$ can take any positive value.

Suppose  that we are given a sample
$X_1,\ldots, X_n$ of  one-dimensional observations from the
model  \eqref{e:the.model}. As in Section \ref{sec:Hill}, we do not use
here the triangular array notation. Neither
the precise value of the tail exponent nor the exact distribution of
the random variables $(R_n)$ in  \eqref{e:the.model} are assumed to be
known. However, we will assume that an upper bound on the tail
exponent $\alpha$ is known.

This section is split into three subsections, describing,
correspondingly, testing the hypothesis of soft truncation,
testing the hypothesis of hard truncation, and testing a slightly
stronger version of the latter.

\subsection{Testing the hypothesis of soft truncation}\label{test1}

We consider the following problem of testing a null hypothesis against
a simple alternative:
\begin{equation}\label{hypo2}
\left.\begin{array}{rcl}
H_0:&P(|H_1|>M)\ll n^{-1}&\mbox{(soft truncation)}\\
H_{\rm alt}:&P(|H_1|>M)\gg n^{-1}&\mbox{(hard truncation)}
\end{array}\right\}.
\end{equation}

We assume the tail exponent $\alpha$ satisfies
\begin{equation} \label{e:upper.b.alpha}
\alpha<A<\infty\,,
\end{equation}
i.e. an upper bound on the tail exponent is available.  As a test
statistic we will use
\begin{equation} \label{e:test.stat.1}
Z_n(A):=\frac{\sum_{i=1}^n|X_i|^A}{\max_{1\le i\le n}|X_i|^A}\,.
\end{equation}
The following proposition describes the asymptotic distribution of $Z_n(A)$
under the null hypothesis and under the alternative.

\begin{prop}\label{prop:test1} (i) \ Under the hypothesis $H_0$
  of soft truncation,
\begin{equation} \label{e:test1.h0}
Z_n(A)\Rightarrow \Gamma_1^{A/\alpha}
\sum_{j=1}^\infty \Gamma_j^{-A/\alpha}\,,
\end{equation}
where $(\Gamma_j,\, j\ge1)$ are the arrival times of a unit rate
Poisson process on $(0,\infty)$.

(ii) \ Assume that $ER_1^A<\infty$. Then under the hypothesis
$H_{\rm alt}$   of hard truncation, $Z_n(A)\stackrel P\longrightarrow\infty$.
\end{prop}
\begin{proof} For part (i), we define
$$
b_n = \inf\bigl\{ x>0:\, P(|H_1|^A>x)\leq n^{-1}\bigr\},\,
n=1,2,\ldots\,.
$$
Note that, for  any $x>0$,
$$
n P\bigl( b_n^{-1}|X_1|^A>x\bigr) \sim n P\bigl(
b_n^{-1}|H_1|^A>x\bigr) \to x^{-\alpha}
$$
as $n\to\infty$. It follows from Proposition 3.21 (page 154) in
\cite{resnick:1987} that we have the following weak convergence
of a sequence of point processes on $(0,\infty]$:
\begin{equation} \label{e:point.pr}
N_n := \sum_{j=1}^n \delta_{b_n^{-1}|X_1|^A} \Rightarrow
N := \sum_{j=1}^\infty \delta_{\Gamma_j^{-A/\alpha}}
\end{equation}
as $n\to\infty$. Here $\delta_a$ is a point mass at $a$, and the
weak convergence takes place in the space of Radon point measures
on $(0,\infty]$  endowed with the topology of vague convergence;
see Section 3.4 in \cite{resnick:1987}. We would like to use the
continuous mapping theorem to deduce \eqref{e:test1.h0} from
\eqref{e:point.pr}, but a preliminary truncation step is
necessary.

For $\vep>0$ we define
$$
Z_n(A;\vep):=\frac{\sum_{i=1}^n|X_i|^A\one\bigl( b_n^{-1}|X_i|^A>\vep
  \bigr)}{\max_{1\le i\le n}|X_i|^A}\,.
$$
Notice that $Z_n(A;\vep) = h(N_n)$, where for a Radon point measure
$\eta= \sum_j \delta_{r_j}$ on $(0,\infty]$,
$$
h(\eta) = \frac{\eta\bigl( (\vep,\infty]\bigr)}{\max_j r_j}\,.
$$
It is standard (and easy) to check that $h$ is continuous with
probability 1 at the Poisson random measure $N$ in \eqref{e:point.pr},
so by the continuous mapping theorem,
$$
Z_n(A;\vep) \Rightarrow \Gamma_1^{A/\alpha}
\sum_{j=1}^\infty \Gamma_j^{-A/\alpha}\one\bigl(
\Gamma_j^{-A/\alpha}>\vep\bigr)\,.
$$
Therefore, the convergence  \eqref{e:test1.h0} will follow once we
check that for every $\delta>0$,
\begin{equation} \label{e:check.1.h0}
\lim_{\vep\to 0}\limsup_{n\to\infty} P\bigl( Z_n(A) -
Z_n(A;\vep)>\delta\bigr)= 0\,.
\end{equation}
To this end, notice that, for any $0<\theta<1$ we can select $\tau>0$
so small that $P\bigl( \max_{1\le i\le n}|X_i|^A\leq \tau b_n\bigr)\leq
\theta$ for all $n$ large enough. Then, for all $n$ large enough,
$$
P\bigl( Z_n(A) - Z_n(A;\vep)>\delta\bigr) \leq \theta
+ \delta^{-1} E\left( \tau^{-1} b_n^{-1} \sum_{i=1}^n|X_i|^A\one\bigl(
  b_n^{-1}|X_i|^A\leq \vep   \bigr)\right)
$$
$$
= \theta + \delta^{-1} \tau^{-1} nb_n^{-1} E\Bigl( |X_1|^A\one\bigl(
  b_n^{-1}|X_1|^A\leq \vep   \bigr)\Bigr)
$$
$$
= \theta + \delta^{-1} \tau^{-1} nb_n^{-1} E\Bigl( |H_1|^A\one\bigl(
  b_n^{-1}|H_1|^A\leq \vep   \bigr)\Bigr)
$$
$$
\sim \theta + \delta^{-1} \tau^{-1} nb_n^{-1}\Bigl( (1-\alpha/A)^{-1}
(\vep b_n) P\bigl( |H_1|^A > \vep b_n\bigr)
$$
$$
\sim \theta + \delta^{-1} \tau^{-1} nb_n^{-1} (1-\alpha/A)^{-1}
(\vep b_n) \bigl( \vep^{-\alpha/A}n^{-1}\bigr)
$$
$$
= \theta + \delta^{-1} \tau^{-1} (1-\alpha/A)^{-1} \vep^{1-\alpha/A}.
$$
where the second equality holds because of soft truncation, and
the first asymptotic equivalence follows from the Karamata
theorem. Since $A>\alpha$, we obtain \eqref{e:check.1.h0} by first
letting $\vep\to 0$ and then $\theta\to 0$. This completes the proof
of part (i).

For part (ii), we start with observing that
\begin{equation} \label{e:testh0.num}
\frac{\sum_{i=1}^n|X_i|^A}{nM_n^A P(|H_1|>M_n)}
 \geq \frac{\sum_{i=1}^n|H_i|^A \one\bigl(M_n/2\leq
|H_i|\le  M_n\bigr)}{nM_n^A P(|H_1|>M_n)}
\end{equation}
$$
\geq (M_n/2)^A \frac{\sum_{i=1}^n \one\bigl(M_n/2\leq
|H_i|\le  M_n\bigr)}{nM_n^A P(|H_1|>M_n)} \sim 2^{-A}(2^\alpha-1)
$$
in probability. On the other hand, for some constant $c>0$, by the
assumption $ER_1^A<\infty$,
$$
\max_{1\le i\le n}|X_i|^A \leq c\bigl( M_n^A + \max_{1\leq j\leq
n}R_j^A\bigr) = cM_n^A + o(1)n
$$
a.s. as $n\to\infty$. Since the truncation is hard, and $A>\alpha$, we
see that
\begin{equation} \label{e:testh0.denom}
\frac{\max_{i=1,\ldots,n}|X_i|^A}{nM_n^A P(|H_1|>M_n)}\to 0
\end{equation}
a.s. as $n\to\infty$ as well. The claim of part (ii) follows from
\eqref{e:testh0.num} and \eqref{e:testh0.denom}.
\end{proof}

Based on Proposition \ref{prop:test1}, we suggest the following test
for the problem \eqref{hypo2}.
\begin{equation} \label{e:test.h0.formal}
\text{reject $H_0$ at significance level $p\in (0,1)$ if} \ \
Z_n(A) >c_p(\alpha/A)\,,
\end{equation}
with  $c_p(\theta)$  such that $P(Z(\theta)>c_p(\theta))=p$, where
for $0<\theta<1$,
\begin{equation} \label{e:test.stat.h0}
Z(\theta) =  \Gamma_1^{1/\theta}
\sum_{j=1}^\infty \Gamma_j^{-1/\theta}\,.
\end{equation}

The random variable $Z(\theta)$ does not seem to have one of the
standard distributions, and we are not aware of any previous studies
of the distribution of $Z(\theta)$. The following proposition lists
some of the properties of this distribution.

\begin{prop}\label{prop:Ztheta}
The random variable $Z(\theta)$ is an infinitely divisible random
variable. It has a density with respect to the Lebesgue measure,
and the Laplace transform
\begin{equation} \label{e:laplace.tr}
Ee^{-\gamma Z(\theta)} = \left( 1 + \gamma e^\gamma\int_0^1
e^{-\gamma x}x^{-\theta}\, dx\right)^{-1}\,,
\end{equation}
$\gamma> \gamma_0$, where $\gamma_0<0$ is the number satisfying
$$
1 + \gamma_0 e^{\gamma_0}\int_0^1
e^{-\gamma_0 x}x^{-\theta}\, dx =0\,.
$$
\end{prop}
\begin{proof}
For $\delta>0$ let
$$
W_\delta = \sum_{j=1}^\infty \bigl( \delta +
\Gamma_j\bigr)^{-1/\theta}\,.
$$
Then $W_\delta$ is an infinitely divisible random variable with the
Laplace transform
$$
Ee^{-\gamma W_\delta} = \exp\left\{ - \int_0^{\delta^{-1/\theta}}
\bigl( 1- e^{-\gamma y}\bigr) \theta y^{-(1+\theta)}\, dy\right\}
$$
for all $\gamma\in\bbr$ because the L\'evy measure of $W_\delta$ has a
compact support; see \cite{rosinski:1990b} and \cite{sato:1999}. Since
$$Z(\theta)\eid1+T^{1/\theta}W_T$$
where $T$ is a standard exponential random variable independent of
$\bigl(\Gamma_j:j\ge1\bigr)$, it follows that
\begin{equation} \label{e:inter.laplace}
Ee^{-\gamma Z(\theta)} = \int_0^\infty e^{-t} e^{-\gamma} Ee^{-\gamma
t^{1/\theta} W_t}\, dt
\end{equation}
$$
= e^{-\gamma} \int_0^\infty e^{-t} \exp\left\{ - t\int_0^1 \bigl( 1-
e^{-\gamma x}\bigr) \theta x^{-(1+\theta)}\, dx\right\}\, dt
$$
$$
= e^{-\gamma} \int_0^\infty\exp\left\{ -t\left[ e^{-\gamma} + \gamma
    \int_0^1 e^{-\gamma x}x^{-\theta}\, dx\right]\right\}\, dt
$$
via integration by parts. Since the exponent under the integral is
positive if and only if $\gamma>\gamma_0$, we obtain
\eqref{e:laplace.tr}. Additionally, it follows from
\eqref{e:inter.laplace} that
\begin{equation} \label{e:mixture}
Z(\theta) \eid 1+Y(T)\,,
\end{equation}
where $\bigl( Y(t),\, t\geq 0\bigr)$ is a subordinator satisfying
\begin{equation}\label{prop:Ztheta:eq1}
Ee^{-\gamma Y(t)} = \exp\left\{ - t\int_0^1 \bigl( 1-
e^{-\gamma x}\bigr) \theta x^{-(1+\theta)}\, dx\right\}, \ \  t\geq
0\,,
\end{equation}
independent of $T$. Since a
L\'evy process stopped at an independent infinitely divisible random
time is, obviously, infinitely divisible, so is
$Z(\theta)$. Furthermore, the characteristic function of $Y(t)$ is
integrable of the real line for every $t>0$, so each $Y(t)$ has a
density, and then the same is true for any mixture of
$(Y(t))$. Therefore, $Z(\theta)$ has a density.
\end{proof}

Even though we know, by Proposition \ref{prop:Ztheta}, that the
random variable $Z(\theta)$ has a density, at present we do not
know ways to compute this density. One possibility to estimate the
critical values $c_p(\alpha/A)$ to perform the test
\eqref{e:test.h0.formal}, is as follows. For values of $\alpha$
not too close to the upper bound $A$ (or, equivalently, for the
values of $\theta$ not too close to $1$), it is possible to
estimate the critical values by the Monte-Carlo method, by
truncating the infinite series at a sufficiently large finite
number of terms. Using $N=10^5$ number of terms in the series and
generating the (truncated) random variable $10^5$ times, we have
estimated the following quantiles, for a range of values $\theta$.

\begin{center}
\begin{tabular}{l|ccc}
\backslashbox{$p$}{$\theta$}&$0.5$&$0.6$&$0.7$\\
\hline
$.05$&$4.3$&$5.8$&$8.2$\\
$.025$&$5.1$&$6.9$&$9.8$\\
$.01$&$6.2$&$8.4$&$12.1$
\end{tabular}
\end{center}

For $\theta$ closer to $1$, the rate of convergence of the truncated
sum $\sum_{j=1}^N\Gamma_j^{-1/\theta}$ as $N\to\infty$
is very slow, and in order to obtain upper bounds on the quantiles of
the random variable $Z(\theta)$ we used Proposition \ref{prop:Ztheta}
as described below. Such upper bounds lead to conservative versions of
the test \eqref{e:test.h0.formal}. We use the exponential Markov
inequality: for $0<r<-\gamma_0$,
$$
P(Z(\theta)\ge z)\le e^{-rz}Ee^{rZ} = e^{-rz} \left( 1 -r
e^{-r}\int_0^1
e^{r x}x^{-\theta}\, dx\right)^{-1}\,,
$$
and estimate the integral from above by
$$
\int_0^1 e^{rx}x^{-\theta}dx \le
e^{r/k}\frac{k^{\theta-1}}{1-\theta}+\frac1k\sum_{j=2}^k
e^{rj/k}\left(\frac{j-1}k\right)^{-\theta}\,,
$$
$k>1$. Using $r=.05$ and $k=10^7$ we computed numbers $\tilde
c_p(\theta)$ satisfying
$$
P\bigl(Z(\theta)\ge\tilde c_p(\theta)\bigr)\le p\,.
$$
These numbers $\tilde c_p(\theta)$  are reported in the following
table.

\begin{center}
\begin{tabular}{l|ccc}
\backslashbox{$p$}{$\theta$}&$0.8$&$0.9$&$0.95$\\
\hline
$.05$&$65.43$&$73.12$&$127.37$\\
$.025$&$79.29$&$86.98$&$141.23$\\
$.01$&$97.62$&$105.31$&$159.56$
\end{tabular}
\end{center}

Since we are only  assuming that the tail exponent
$\alpha$ has a known upper bound as in \eqref{e:upper.b.alpha}, but
the exact value of $\alpha$ may be unknown, a possible way to obtain a
conservative estimate of the critical value $c_p(\alpha/A)$ in
\eqref{e:test.stat.h0} is to choose a number $A_1>A$ and use the
statistic $Z_n(A_1)$ instead of $Z_n(A)$ in
\eqref{e:test.stat.1}. By Proposition \ref{prop:test1},
under the null hypothesis, the test statistic
converges weakly to $Z(\alpha/A_1)$, which is stochastically smaller
than $Z(A/A_1)$, and we  obtain a conservative test by
modifying \eqref{e:test.h0.formal} as follows:
\begin{equation} \label{e:test.h0.mod}
\text{reject $H_0$ at significance level $p\in (0,1)$ if} \ \
Z_n(A_1) >c_p(A/A_1)\,.
\end{equation}

\subsection{Testing the hypothesis of hard truncation}\label{test2}

In this subsection we consider the following problem of testing a null
hypothesis against a simple alternative:

\begin{equation}\label{hypo1}
\left.\begin{array}{rcl}
H_0:&P(|H_1|>M)\gg n^{-1}&\mbox{(hard truncation)}\\
H_{\rm alt}:&P(|H_1|>M)\ll n^{-1}&\mbox{(soft truncation)}
\end{array}\right\}.
\end{equation}

We still assume that an upper bound \eqref{e:upper.b.alpha} on the
tail exponent is known. For a test statistic in this case we
choose a number $\gamma\in(0,1)$ and define
\begin{equation} \label{e:test.stat.2}
Z_n(A;\gamma)=\frac{\left(\sum_{j=1}^{[\gamma
      n]}(-1)^jX_{j}^{\langle A/2\rangle}\right)^2}{\sum_{j=[\gamma
      n]+1}^n|X_{j}|^A}.
\end{equation}
Here $a^{\langle b\rangle}= |a|^b{\rm sign}(a)$ for real $a,b$ is the
signed power. The asymptotic distribution of $Z_n(A;\gamma)$
under the null hypothesis and under the alternative in \eqref{hypo1}
is described in Proposition \ref{prop:test2} below. Recall the
standard notation of $S_\alpha(\sigma,\beta,\mu)$ for (the
distribution of) an
$\alpha$-stable random variable with the scale $\sigma$, skewness
$\beta$ and location $\mu$; see \cite{samorodnitsky:taqqu:1994}. For a
symmetric $\alpha$-stable random variable, $\beta=\mu=0$. For a
positive strictly $\alpha$-stable random variable with $0<\alpha<1$,
one has $\beta=1$ and $\mu=0$. Finally, for $0<\alpha<2$, let
$$
C_\alpha = \begin{cases}
(\Gamma(1-\alpha) \cos (\pi\alpha/2))^{-1}
%\frac{1-\alpha}{\Gamma(2-\alpha) \cos (\pi\alpha/2)}
 &\mbox{if $\alpha\neq 1,$}\\
%\frac{2}{\pi}
2/\pi &\mbox{if $\alpha =1,$}
\end{cases}
$$

\begin{prop}\label{prop:test2} (i) \ Assume that $ER_1^{2A}<\infty$. Then
  under the hypothesis $H_0$   of hard truncation,
\begin{equation} \label{e:test2.h0}
Z_n(A;\gamma)\Rightarrow C_1(\gamma) \chi_1^2\,,
\end{equation}
where $C_1(\gamma)=2\gamma/(1-\gamma)$,
and $\chi_1^2$ is the standard chi-square random variable with one
degree of freedom.

(ii) \ Under the hypothesis $H_{\rm alt}$ of soft truncation,
\begin{equation} \label{e:test2.h1}
Z_n(A;\gamma)\Rightarrow C_2(A;\gamma) \frac{S_1^2}{S_2}\,,
\end{equation}
where
$$
C_2(A;\gamma)=\left(
\frac{\gamma}{1-\gamma}\frac{C_{\alpha/A}}{C_{2\alpha/A}}
\right)^{A/\alpha}\,,
$$
and $S_1$ and $S_2$ are independent
random variables, such that $S_1$ is a symmetric $2\alpha/A$-stable
random variable with unit scale, and $S_2$ is a positive strictly
$\alpha/A$-stable random variable with unit scale.
\end{prop}
\begin{proof}
The claim of part (i) will follow from the following two statements.
\begin{equation} \label{e:test2.num}
\frac{1}{(nM_n^A P(|H_1|>M_n))^{1/2}}
\sum_{j=1}^{[\gamma
      n]}(-1)^jX_{j}^{\langle A/2\rangle} \Rightarrow \left(
      \frac{2A\gamma}{A-\alpha}\right)^{1/2} N(0,1)\,,
\end{equation}
and
\begin{equation} \label{e:test2.denom}
\frac{1}{nM_n^A P(|H_1|>M_n)} \sum_{j=[\gamma
      n]+1}^n|X_{j}|^A \to \frac{A(1-\gamma)}{A-\alpha}
\end{equation}
in probability. We prove \eqref{e:test2.denom} first, and it is enough
to show that
\begin{equation} \label{e:test2.denom1}
\frac{1}{nM_n^A P(|H_1|>M_n)}E\left( \sum_{j=[\gamma
      n]+1}^n|X_{j}|^A\right)\to \frac{A(1-\gamma)}{A-\alpha}
\end{equation}
and
\begin{equation} \label{e:test2.denom2}
\frac{1}{\bigl( nM_n^A P(|H_1|>M_n)\bigr)^2}
{\rm Var}\left( \sum_{j=[\gamma
      n]+1}^n|X_{j}|^A\right)\to 0\,.
\end{equation}
Note that by the Karamata theorem,
$$
E\left( \sum_{j=[\gamma n]+1}^n|X_{j}|^A\right)
\sim (1-\gamma)n\, E\bigl( |X_{1}|^A\bigr)
$$
$$
= (1-\gamma)n \Bigl[ E\bigl( |H_1|^A \one(|H_1|\leq M_n)\bigr)
+ E(M_n+R_1)^A P(|H_1|> M_n)\Bigr]
$$
$$
\sim (1-\gamma)n \left[ \frac{\alpha}{A-\alpha}M_n^A P(|H_1|>M_n)
+ M_n^A P(|H_1|>M_n)\right]
$$
$$
= \bigl( nM_n^A P(|H_1|>M_n)\bigr)
\frac{A(1-\gamma)}{A-\alpha} \,,
$$
proving \eqref{e:test2.denom1}. A similar calculation gives us
$$
{\rm Var}\left( \sum_{j=[\gamma n]+1}^n|X_{j}|^A\right)
\sim (1-\gamma)n\, {\rm Var}\bigl( |X_{1}|^A\bigr)
$$
$$
\leq n\, E\bigl( |X_{1}|^{2A}\bigr)
\sim \bigl( nM_n^{2A} P(|H_1|>M_n)\bigr)\frac{2A}{2A-\alpha}\,,
$$
and \eqref{e:test2.denom2} follows because the truncation is
hard. Therefore, we have established \eqref{e:test2.denom}.

In order to prove \eqref{e:test2.num}, note that the triangular
array
$$
\tilde X_{nj}:=H_j^{\langle A/2\rangle}\one\bigl(|H_j|^{A/2}\le
M_n^{A/2}\bigr)
+\frac{H_j}{|H_j|}(M_n^{A/2}+R_j^{A/2})\one\bigl(|H_j|^{A/2}>M_n^{A/2}
\bigr)\,,
$$
$j=1,\ldots, n, \, n=1,2,\ldots$, satisfies the assumptions of Theorem
\ref{t2} (with $\alpha$ replaced by $2\alpha/A$), and, therefore,
$$
\frac{1}{ \bigl( nM_n^A P(|H_1|>M_n)\bigr)^{1/2}}\left( \sum_{j=1}^n
\tilde X_{nj} -
E \Bigl(  \sum_{j=1}^n \tilde X_{nj}\Bigr)\right)
\Rightarrow \left(\frac{2A}{A-\alpha}\right)^{1/2} N(0,1)\,.
$$
The random variables $\bigl( X_j^{\langle A/2\rangle}\bigr)$ form a
somewhat different triangular array, namely
$$
X_{nj}^{\langle A/2\rangle}=H_j^{\langle
A/2\rangle}\one\bigl(|H_j|^{A/2}\le M_n^{A/2}\bigr)
+\frac{H_j}{|H_j|}(M_n+R_j)^{A/2}\one\bigl(|H_j|^{A/2}>M_n^{A/2}
\bigr)\,,
$$
$j=1,\ldots, n, \, n=1,2,\ldots$, but an inspection of the proof
of Theorem \ref{t2} shows that the argument applies equally well
to the latter triangular array, so that
$$
\frac{1}{ \bigl( nM_n^A P(|H_1|>M_n)\bigr)^{1/2}}\left( \sum_{j=1}^n
X_{nj}^{\langle
A/2\rangle} -
E \Bigl(  \sum_{j=1}^n X_{nj}^{\langle A/2\rangle}\Bigr)\right)
$$
$$
\Rightarrow \left(\frac{2A}{A-\alpha}\right)^{1/2} N(0,1)\,.
$$
In particular, (extending the length of the rows of the triangular
array) we see that
$$
\frac{1}{ \left( nM_n^A P(|H_1|>M_n)\right)^{1/2}}\left( \sum_{j=1}^n
X_{nj}^{\langle A/2\rangle} - \sum_{j=n+1}^{2n} X_{nj}^{\langle
A/2\rangle}\right)
$$
$$
\Rightarrow \left(\frac{4A}{A-\alpha}\right)^{1/2} N(0,1)\,.
$$
Replacing $n$ with $[n\gamma/2]$, we obtain \eqref{e:test2.num} and,
hence, finish the proof of  part (i).

For part (ii), we define
$$
b_n = \inf\bigl\{ x>0:\, P(|H_1|^{A/2}>x)\leq n^{-1}\bigr\},\,
n=1,2,\ldots \,.
$$
Then for some centering sequence $(c_n)$ we have
$$
b_n^{-1}\left(\sum_{j=1}^nH_j^{\langle A/2\rangle}-c_n\right)
\Rightarrow Y
$$
with $Y$ having a $S_{2\alpha/A}(\sigma,\beta,\mu)$ distribution with
$\sigma^{2\alpha/A} = (C_{2\alpha/A})^{-1}$ and some $\beta,\mu$; see
\cite{feller:1971}. Because of the soft truncation, the triangular
array $\bigl( X_{nj}^{\langle A/2\rangle}\bigr)$ satisfies Theorem
\ref{t1}, and so
$$
b_n^{-1}\left(\sum_{j=1}^nX_{nj}^{\langle A/2\rangle}-c_n\right)
\Rightarrow Y
$$
with the same $Y$. Extending the rows of the triangular array gives us
$$
b_n^{-1}\left(\sum_{j=1}^nX_{nj}^{\langle
A/2\rangle}-\sum_{j=n+1}^{2n}X_{nj}^{\langle A/2\rangle} \right)
\Rightarrow \left( \frac{2}{C_{2\alpha/A}}\right)^{A/(2\alpha)}S_1\,,
$$
where $S_1$ is a symmetric $2\alpha/A$-stable
random variable with unit scale. Replacing $n$ with $[n\gamma/2]$
we obtain
\begin{equation} \label{e:t2:soft.1}
\sum_{j=1}^{[\gamma
      n]}(-1)^jX_{j}^{\langle A/2\rangle} \Rightarrow
\left( \frac{\gamma}{C_{2\alpha/A}}\right)^{A/(2\alpha)}S_1\,.
\end{equation}

Next, we also have
$$
b_n^{-2}\sum_{j=1}^n|H_{j}|^A \Rightarrow
\left(\frac{1}{C_{\alpha/A}}\right)^{A/\alpha}S_2\,,
$$
where $S_2$ is a positive strictly
$\alpha/A$-stable random variable with unit scale; see once again
\cite{feller:1971}. As before, because of the soft truncation, Theorem
\ref{t1} applies, and we obtain
$$
b_n^{-2}\sum_{j=1}^n|X_{nj}|^A \Rightarrow
\left(\frac{1}{C_{\alpha/A}}\right)^{A/\alpha}S_2\,.
$$
Replacing $n$ with $(1-\gamma)n$, shows that
\begin{equation} \label{e:t2:soft.2}
b_n^{-2}\sum_{j=[\gamma n]+1}^n|X_{j}|^A \Rightarrow
\left(\frac{1-\gamma}{C_{\alpha/A}}\right)^{A/\alpha}S_2\,.
\end{equation}
Since the numerator and the denominator of the statistic
$Z_n(A;\gamma)$ in \eqref{e:test.stat.2} are independent, the claim of
part (ii) of the proposition follows from \eqref{e:t2:soft.1} and
\eqref{e:t2:soft.2}.
\end{proof}

Interestingly, the asymptotic distribution of the test statistic
$Z_n(A;\gamma)$, under the null hypothesis, does not depend on the
choice of the parameter $A$ (as long as it an  upper bound on the tail
exponent
$\alpha$). Furthermore, under the null hypothesis this asymptotic
distribution of the test statistic is
light-tailed (e.g. some  exponential moments are finite).  On the
other hand, the asymptotic distribution of the test statistic
under the alternative is, clearly, heavy tailed, as even the second
moment is infinite. Therefore, a reasonable test will reject the null
hypothesis in favor of the alternative if the test statistic is too
large. That is, we suggest the following test
for the problem \eqref{hypo1}.
\begin{equation} \label{e:test.h1.formal}
\text{reject $H_0$ at significance level $p\in (0,1)$ if} \ \
Z_n(A;\gamma) >\frac{2\gamma}{1-\gamma}c_p\,,
\end{equation}
with $c_p$ such that $P(\chi_1^2>c_p)=p$.

\subsection{Testing a stronger version of the hypothesis of hard
truncation}\label{test3}

The test statistics $Z_n(A;\gamma)$ we used in the previous subsection
for the problem \eqref{hypo1} has a nondegenerate asymptotic
distribution under both the null hypothesis and the alternative. This
may restrict the power of the resulting test. In order to obtain a
more powerful test we strengthen the null hypothesis. Specifically,
in this subsection we consider the following problem of testing a null
hypothesis against a simple alternative:
\begin{equation} \label{hypo3}
\left. \begin{array}{rc}
H_0:&n^{1-\epsilon}P(|H_1|>M)\gg1\\
H_{\rm alt}:&nP(|H_1|>M)\ll1
\end{array} \right\},
\end{equation}
where $\epsilon$ is a fixed number in $(0,1)$.

For this problem one can use the same test statistic $Z_n(A)$ defined
in \eqref{e:test.stat.1} as we used for the problem \eqref{hypo2} of
testing the hypothesis of soft truncation. Proposition
\ref{prop:test1} tells us that this test statistic diverges in
probability to infinity under the hypothesis of hard truncation.
The strengthened hypothesis of hard truncation in \eqref{hypo3} allows
us to quantify how fast this divergence takes place. This, in turn,
can be used to build a test. The asymptotic distribution of $Z_n(A)$
under the hypothesis of soft truncation is described in Proposition
\ref{prop:test1}. The next result provides an asymptotic
distributional lower bound on the test statistic under the null
hypothesis in the problem \eqref{hypo3}. As in the previous
subsections, we assume that an upper bound \eqref{e:upper.b.alpha} on the
tail exponent is known.

\begin{prop}\label{prop:test3} \ Assume that $ER_1^{2A}<\infty$. Then
  under the strengthened hypothesis $H_0$  of hard truncation,
\begin{equation} \label{e:test3.h0}
\liminf_{n\to\infty} P\Bigl( n^{-\epsilon/2}Z_n(A)> x\Bigr) \geq
e^{-x^{2}}
\end{equation}
for every $x>0$.
\end{prop}
\begin{proof}
In the notation of the triangular array \eqref{e:the.model}, consider
the binomial random variable
$N_n = \sum_{j=1}^n \one\bigl(|H_j|>M_n\bigr)$. The strengthened
hypothesis of hard truncation implies that $P(N_n\geq n^{\epsilon})\to
1$ as $n\to\infty$. Notice that, on an event of probability increasing
to 1,
$$
Z_n(A) \geq \frac{ \sum_{j=1}^n
(M_n+R_j)^A\one\bigl(|H_j|>M_n\bigr)}{\max_{j=1,\ldots,n}
(M_n+R_j)^A\one\bigl(|H_j|>M_n\bigr)}
$$
$$
\geq \frac{ \sum_{j=1}^n
R_j^A\one\bigl(|H_j|>M_n\bigr)}{\max_{j=1,\ldots,n}
R_j^A\one\bigl(|H_j|>M_n\bigr)}\,.
$$
Therefore, for $x>0$, using the assumption $ER_1^{2A}<\infty$, we have
$$
\liminf_{n\to\infty} P\Bigl( n^{-\epsilon/2}Z_n(A)> x\Bigr) \geq
\liminf_{n\to\infty}P\Bigl(\max_{j=1,\ldots,
N_n}R_j^A < n^{-\epsilon/2}N_n\frac{ER_1^A}{2}x^{-1}\Bigr)
$$
$$
\geq \liminf_{n\to\infty} E\left[ \left(
1-\frac{x^2}{N_n}\right)^{N_n}\one\Bigl( N_n\geq
n_\epsilon\Bigr) \right]
\to e^{-x^2}\,,
$$
as required.
\end{proof}

Proposition \ref{prop:test3} tells us that under the hypothesis $H_0$,
$n^{-\epsilon/2}Z_n(A)$ is, asymptotically, stochastically larger than
the square root of the standard exponential random variable
(independently of the parameter $A$). Therefore, we suggest the
following test for the problem \eqref{hypo3}.
\begin{equation} \label{e:test.h3.formal}
\text{reject $H_0$ at significance level $p\in (0,1)$ if} \ \
Z_n(A) \leq \bigl| \log(1-p)\bigr|^{1/2}n^{\epsilon/2}\,.
\end{equation}

\section{Think Times and Object Sizes data:  soft truncation or hard 
  truncation? }  \label{sec:data.analysis}

In this section we applied the statistical methods of Section
\ref{sec:regime.test} to two data sets. One data set contains
``think times'', or delays (in microseconds) between successive
request/response exchanges between hosts using a TCP connection.
The second data set contains the sizes (in bytes) of objects
(files, HTTP responses, email messages, etc.) transferred on TCP
connections. Both data sets were acquired by monitoring between
1:30 PM and 2:30 PM on July 24, 2006, the communication links
connecting the site of a large commercial enterprize to the
Internet. Both data sets exhibit visual evidence of heavy tails,
and the Hill estimator confirms that (see below).  Our goal is to
check if the data sets show statistical evidence of soft or hard 
truncation of heavy tails. As we will see, in a number of cases we are
not able to reject either the hypothesis of soft truncation or that of
hard truncation. This is an indication that there is much room for
improving the statistical techniques of testing for the type of the
truncation of the tails that provides the best approximation to
data. We will comment more on that issue at the end of the section. 

\subsection{Think Times} This data sets contains $2.1\times10^7$
observations which are plotted on Figure \ref{data:ThinkTime}.
\begin{figure}
\begin{center}
\includegraphics[width=8cm]{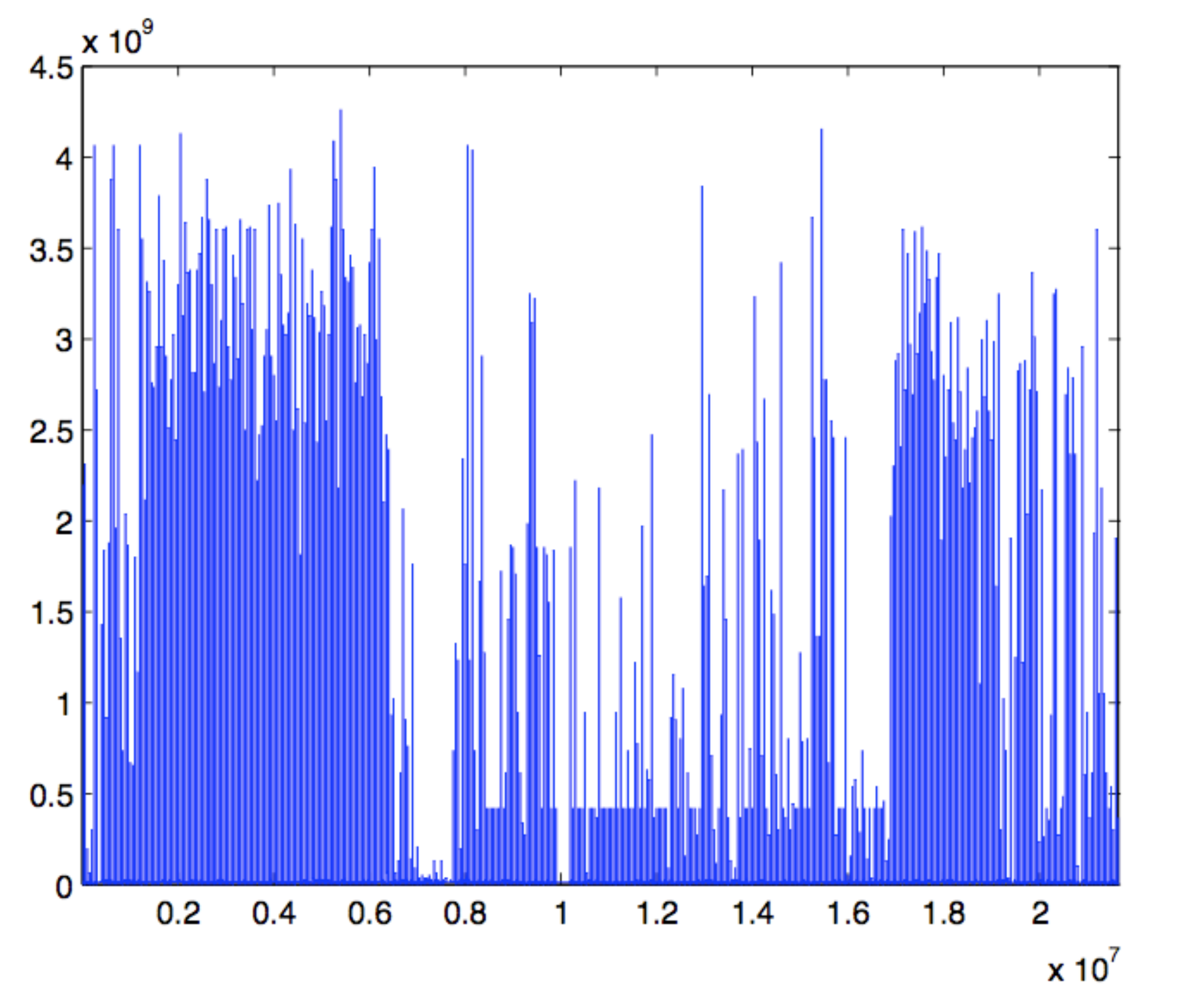}
\caption{Think Times - the entire data set}\label{data:ThinkTime}
\end{center}
\end{figure}

Clearly, the nature of this data set changes over time, and the nature
of truncation of heavy tails may potentially change as well.
In order to study this effect we have, somewhat arbitrarily, broken
the data set into four pieces, with corresponding ranges
$\bigl[ 0.11\times 10^7, 0.64\times10^7\bigr]$;
$\bigl[ 0.8\times 10^7, 1.6\times10^7\bigr]$;
$\bigl[ 1.7\times 10^7, 1.9\times10^7\bigr]$ and
$\bigl[ 1.95\times 10^7, 2.1\times10^7\bigr]$.
The individual pieces are plotted on Figure \ref{data:ThinkTimes_pieces}
\begin{figure}
\begin{center}
\includegraphics[width=12cm]{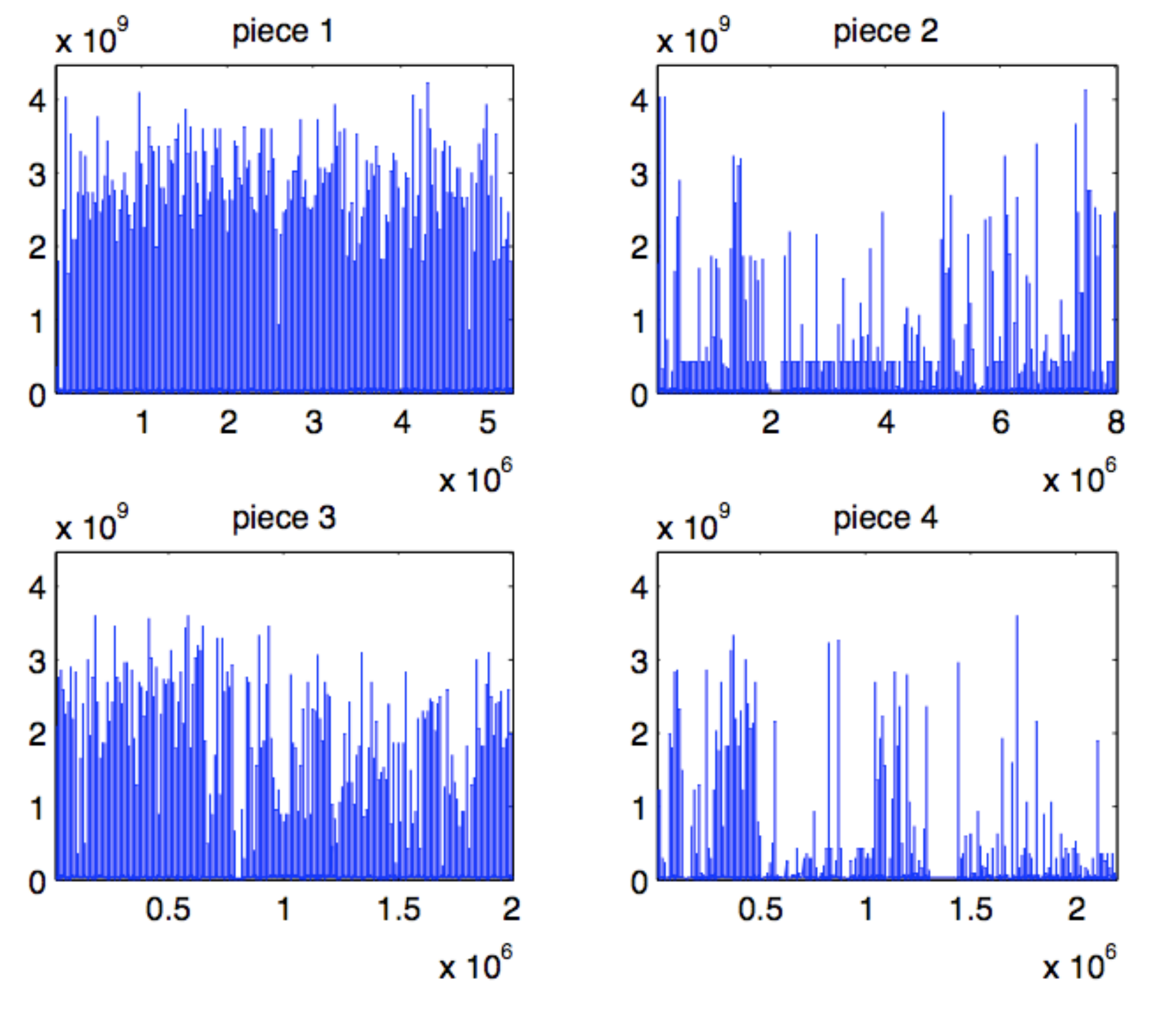}
\caption{Think Times - the different pieces}\label{data:ThinkTimes_pieces}
\end{center}
\end{figure}

The structure of the 4 individual pieces appears to be more stable
than that of the entire data sets, and we proceed to analyze each
piece separately. To do that, we first ran the Hill estimator with
random $k$ given in \eqref{e:k.random} on the first half of each
of the 4 pieces. The estimation was conducted using $\beta,
\gamma=0.3,0.4, 0.5, 0.6, 0.7$ and conservative upper bounds for
$\alpha$ were obtained; these are presented in the following
table.

\begin{center}
\begin{tabular}{c|c}
piece & $A$\\
\hline
1 & $3.02$\\
2 & $2.30$\\
3 & $0.85$\\
4 & $2.24$
\end{tabular}
\end{center}
We then proceeded to use the second halves of each piece of the Think
Times data set to test for soft and hard truncations.

\subsubsection*{Testing the hypothesis of soft truncation}
The test statistic $Z_n(A_1)$ of Section
\ref{test1}  was computed for
various values of $A_1$ larger than $A$. The results are reported in
the following table.
\begin{center}
\begin{tabular}{c|cccc}
$A/A_1$&piece 1&piece 2&piece 3&piece 4\\
\hline
$0.5$&$31.43$&$5.81$&$154.05$&$3.57$\\
$0.6$&$51.59$&$7.99$&$205.37$&$4.72$\\
$0.7$&$77.39$&$10.74$&$271.27$&$6.11$\\
$0.8$&$108.08$&$14.20$&$361.74$&$7.81$\\
$0.9$&$142.78$&$18.57$&$491.31$&$9.91$\\
$0.95$&$161.38$&$21.16$&$576.73$&$11.13$
\end{tabular}
\end{center}
Comparing the resulting values of the test statistic  with the
corresponding quantiles (or their upper bounds) of $Z(A/A_1)$, it is
clear that the null hypothesis
of soft truncation can be rejected for pieces 1 and 3 at the
levels less that $0.01$, while for piece
2 the null hypothesis of soft truncation can be rejected at the level
$0.025$, but not at the level $0.01$. For piece 4 we cannot reject the
null hypothesis of soft truncation even at the level $0.05$. These
findings are rather consistent with visual analysis of the pieces of
the data set: piece 4 exhibits extremes that do not appear to be
seriously truncated (one has to keep in mind that the visual analysis
can easily be misleading). 

\subsubsection*{Testing the hypothesis of hard truncation}
The test statistic $Z_n(A;\gamma)$ of Section
\ref{test2}  was computed for various values of $\gamma$.
The resulting $p$-values are reported in the following table.
\begin{center}
\begin{tabular}{c|cccc}
$\gamma$&piece 1&piece 2&piece 3&piece 4\\
\hline
$0.1$&$0.85$&$0.72$&$0.88$&$0.33$\\
$0.2$&$0.83$&$0.98$&$0.38$&$0.57$\\
$0.3$&$0.97$&$0.99$&$0.79$&$0.68$\\
$0.4$&$0.94$&$0.68$&$0.39$&$0.43$\\
$0.5$&$0.83$&$0.63$&$0.94$&$0.47$\\
$0.6$&$0.97$&$0.89$&$0.83$&$0.27$\\
$0.7$&$0.91$&$0.88$&$0.87$&$0.40$\\
$0.8$&$0.64$&$0.85$&$0.80$&$0.33$\\
$0.9$&$0.70$&$0.37$&$0.85$&$0.40$
\end{tabular}
\end{center}
The hypothesis of hard truncation cannot be rejected for
any of the four pieces. As expected, the $p$-values for piece 4 are
lower than those for the other 3 pieces, but they are still rather
high.

\subsubsection*{Testing  a stronger version of the hypothesis of hard
  truncation}

The test statistics $Z_n(A)$ of Section \ref{test3} was computed
and the corresponding $p$-values calculated for various values of
$\epsilon$. These are listed in the following table.
\begin{center}
\begin{tabular}{c|cccc}
$\epsilon$&piece 1&piece 2&piece 3&piece 4\\
\hline
$0.1$&$1.00$&$1.00$&$1.00$&$1.00$\\
$0.2$&$1.00$&$1.00$&$1.00$&$1.00$\\
$0.3$&$1.00$&$1.00$&$1.00$&$0.91$\\
$0.4$&$1.00$&$0.73$&$1.00$&$0.45$
\end{tabular}
\end{center}

It is clear that even the stronger version of the  hypothesis of hard
truncation cannot be rejected.

\subsection{Object Sizes} This data set contains $2.2\times10^7$
observations. It is plotted in Figure
\ref{data:ObjectSize}. It does not appear that the nature of the
observations changes with time, so we applied our statistical tests to
the entire data set. After running the Hill estimator with random $k$
and parameters $\beta$ and $\gamma$ as above, on the first half of the
data set, we obtained a conservative upper bound on the value of the
tail exponent $\alpha$; this turned out to be $A=1.69$. We used the
second half of the Object Sizes data set to test for soft and hard
truncations.
\begin{figure}
\begin{center}
\includegraphics[width=8cm]{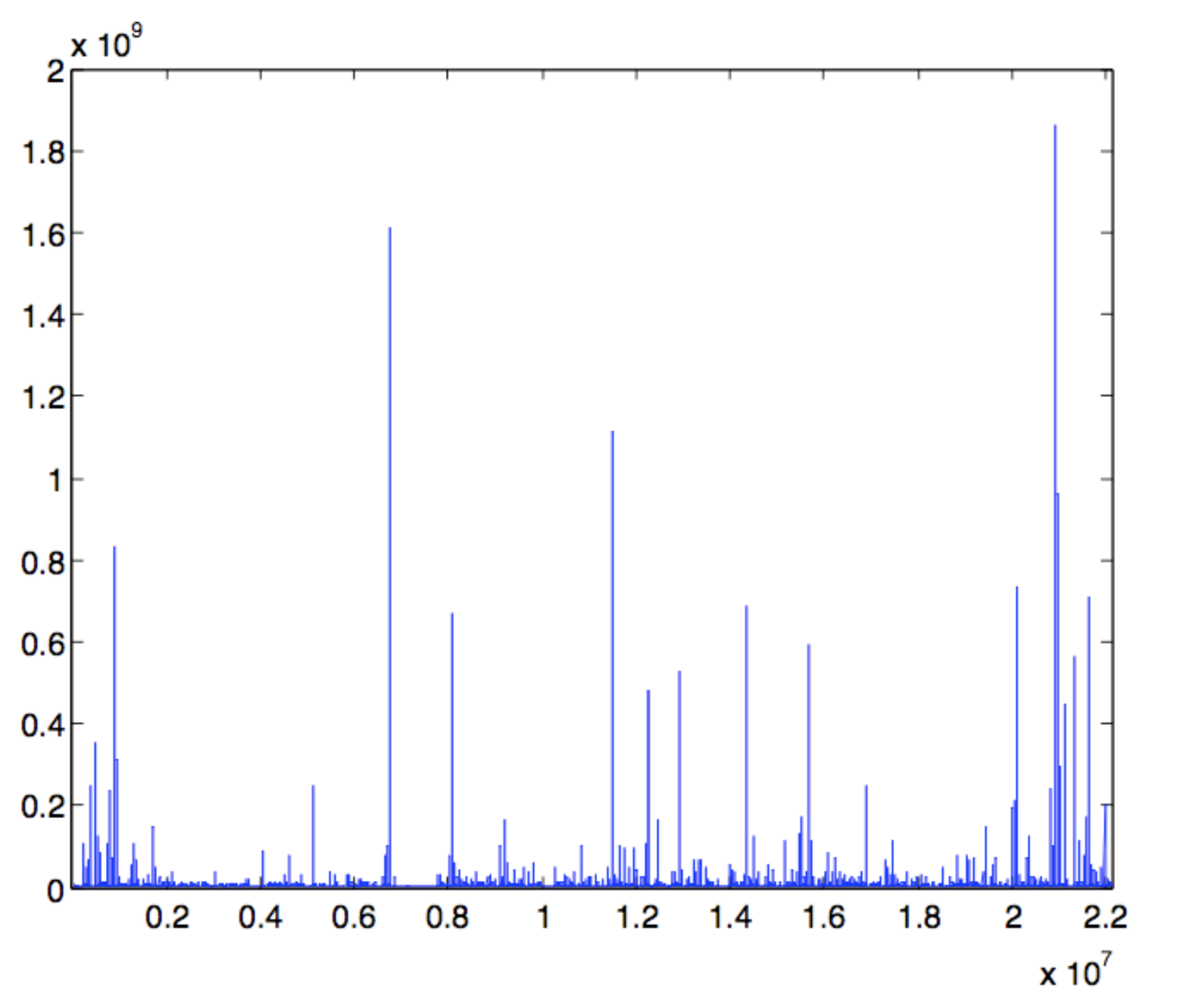}
\caption{Data on Object Sizes}\label{data:ObjectSize}
\end{center}
\end{figure}

\subsubsection*{Testing the hypothesis of soft truncation} We
evaluated the test statistic $Z_n(A_1)$ of Section
\ref{test1}  for a range of values of $A_1$ larger than $A$. The
results are reported in the following table.
\begin{center}
\begin{tabular}{c|c}
$A/A_1$&$Z_n(A_1)$\\
\hline
$0.5$&$1.75$\\
$0.6$&$2.32$\\
$0.7$&$3.09$\\
$0.8$&$4.10$\\
$0.9$&$5.42$\\
$0.95$&$6.23$
\end{tabular}
\end{center}
Comparing these with the corresponding quantiles (or their upper
bounds) of $Z(A/A_1)$, we see that the hypothesis of soft truncation
cannot be rejected.

\subsubsection*{Testing the hypothesis of hard truncation} We
evaluated the test statistic $Z_n(A;\gamma)$ of Section
\ref{test2}  for various values of $\gamma$, and the obtained $p$-values
are reported in the following table.
\begin{center}
\begin{tabular}{c|c}
$\gamma$&$p$-value\\
\hline
$0.1$&$0.50$\\
$0.2$&$0.36$\\
$0.3$&$0.73$\\
$0.4$&$0.77$\\
$0.5$&$0.95$\\
$0.6$&$0.94$\\
$0.7$&$0.94$\\
$0.8$&$0.97$\\
$0.9$&$0.72$
\end{tabular}
\end{center}
The null hypothesis of hard truncation cannot be rejected.

\subsubsection*{Testing  a stronger version of the hypothesis of hard
  truncation}

We calculated the test statistics $Z_n(A)$ of Section \ref{test3} for
various values of $\epsilon$, and the
$p$-values are given in the following table.
\begin{center}
\begin{tabular}{c|c}
$\epsilon$&$p$-value\\
\hline
$0.1$&$1.00$\\
$0.2$&$0.86$\\
$0.3$&$0.33$\\
$0.4$&$0.08$
\end{tabular}
\end{center}
The strengthened hypothesis of hard   truncation becomes suspicious
for $\epsilon=0.4$, but overall our statistical tests do not produce
clear evidence of the level of truncation for the Object Sizes data
set.

\bigskip

{\bf Summary} \ The results of this section, especially the lack of
conclusion in a number of cases, point to some of the
directions  for future statistical work with truncated heavy tails.
\begin{itemize}
\item More powerful tests are needed, together with estimates of their 
actual power. 
\item What sample sizes are necessary for the asymptotic
significance levels to be applicable?
\item The tests of Section \ref{sec:regime.test} depend on tuning
parameters $A$ and $\gamma$. At present we do not have a clear picture
of the role of these parameters, or how to select them in the 
``optimal'' way. 
\item What is the best way to select the parameters $\beta$ and
$\gamma$ in the sample-based number of the upper order statistics in
\eqref{e:k.random} to use in Hill estimation? Can one combine
estimation of the tail index with the testing for a truncation regime
in a single procedure?
\end{itemize}

\section{Acknowledgment} The authors wish to thank Dr. F. Donelson
Smith of the Network Research Laboratory in the Computer Science
Department at the University of North Carolina at Chapel Hill for kindly
providing the data sets analyzed in Section
\ref{sec:data.analysis}.

%\bibliography{C:/Mydata/work_arijit/res_ht/bibfile}
%\bibliographystyle{C:/Mydata/work_arijit/res_ht/apalike}

\bibliography{F:/work_arijit/res_ht/bibfile}
\bibliographystyle{F:/work_arijit/res_ht/apalike}

%\bibliographystyle{z:/Gena/GenaFiles/texfiles/mystyle}
%\bibliography{z:/Gena/GenaFiles/texfiles/bibfile}

\end{document}